\documentclass[11pt,a4paper,reqno]{amsart}
\usepackage{graphicx}
\usepackage{tikz}

\usepackage{amsmath}%
\usepackage{amssymb,indentfirst,calc,mathrsfs,euscript,bbm}%
\usepackage{dsfont}
\usepackage{enumerate}
\usepackage{setspace}%
\usepackage[reals]{layout}%
\usepackage{xr}%
\usepackage{amscd}%
\usepackage{comment}

\def\rit{\mathbb R}

\def\nit{\hbox{\it I\hskip -2pt  N}}

\def\cB{{\mathcal B}}

\def\bof{{\bf f}}

\def\bol{{\bf l}}

\def\bz{{\bf z}}

\newcommand{\ind}{\ensuremath{\mathds 1}}
\newcommand{\inte}{\ensuremath{\operatorname{int}}}

\newtheorem{defi}{Definition}[section]
\newtheorem{theo}{Theorem}[section]

\newtheorem{cor}{Corollary}[section]
\newtheorem{pro}{Proposition}[section]

\usetikzlibrary{shapes}

\begin{document}

\title{Finite composite games: equilibria and dynamics}

\author{
Sylvain Sorin}

\address{
\textbf{Sylvain Sorin}\\
Sorbonne Universit\'es, UPMC Univ Paris 06, Institut de Math\'ematiques de Jussieu-Paris Rive Gauche, UMR 7586, CNRS, Univ Paris Diderot, Sorbonne Paris Cit\'e, F-75005, Paris, France}

\email{
sylvain.sorin@imj-prg.fr}

\thanks{}

\author{Cheng Wan}

\address{\textbf{Cheng Wan}\\
(Corresponding author) Department of Economics, University of Oxford; Nuffield College, New Road, Oxford, OX1 1NF, United Kingdom}

\email{
cheng.wan.2005@polytechnique.org\newline
https://sites.google.com/site/wancheng2012avril/
}
 \thanks{}

\date{October 21, 2015}

\subjclass[2010]{91A10, 91A22, 91A06, 91A13}

\keywords{finite game, composite game, variational inequality, potential game, dissipative game, evolutionary dynamics, Lyapunov function}

\thanks{The first author was partially supported by PGMO 2014-COGLED. We thank the referee for his or her very precise reading and comments.}

\bibliographystyle{apalike}

\begin{abstract} We study games with finitely many participants, each having finitely many choices.
We consider the following categories of participants:\\
(I) populations: sets of nonatomic agents,\\
(II) atomic splittable players,\\
(III) atomic non splittable players.\\
We recall  and compare the basic properties, expressed through variational inequalities, concerning equilibria, potential games and dissipative games, as well as evolutionary  dynamics.\\
Then we consider composite games where the three categories of participants are present, a typical example being congestion games,  and extend the previous properties of equilibria and dynamics.\\
Finally we describe an instance of composite potential game.
\end{abstract}

\maketitle

\section{Introduction}

We study equilibria and dynamics in finite games: there are finitely many ``participants" $i\in I$ and each of them has finitely many ``choices" $p\in S^i$.
The basic variable describing the strategic interaction  is thus a profile  $x = \{ x^{i}, \,  i \! \in\! I\}$, where each $x^i = \{x^{i}_p, \, p \!\in\! S^i\}$ is an element of the simplex $X^i = \Delta(S^i)$ on $S^i$. Let $X = \prod_{i\in I} X^i$.
\medskip

We consider three frameworks (or categories):
\smallskip

\noindent ($\mathbf{I}$) Population games where each  participant  $i \in I$ corresponds to a population: a nonatomic set of agents having all the same characteristics. In this setup $x^{i}_p$ is the proportion of agents of  ``type $p$" in population $i$. 
\smallskip

The two others correspond to  two kinds of $I$-player games where each participant $i\in I$ stands for an atomic player:

\noindent ($\mathbf{I\!I}$) Splittable case: $x^{i}_p$ is the proportion that player $i$ allocates to  choice $p$. (The set of pure moves of player $i$ is $X^i$.) 

\noindent ($\mathbf{I\!I\!I}$) Non splittable case: $x^{i}_p$ is the probability that player $i$ chooses  $p$. (The set of pure moves is $S^i$ and $x^i$ is a mixed strategy.) 
\medskip

As an example, consider  the following network where a routing game of each of the three frameworks takes place. Assume that arc $1$ and arc $2$ are connecting $o$ to $d$.
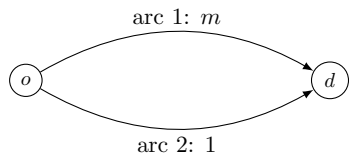
\begin{figure}[htb!]
 \begin{center}
 \begin{tikzpicture}
 \node[draw,circle,scale=0.7] (O)at(-2,0) {$o$};
 \node[draw,circle,scale=0.7] (D)at(2,0) {$d$};
 \draw[->,>=latex] (O) to[bend left=30] node[midway,above,scale=0.8]{arc 1: $m$}(D);
 \draw[->,>=latex] (O) to[bend right=30] node[midway,below,scale=0.8]{arc 2: $1$}(D);
 \end{tikzpicture}
 \end{center}
\end{figure}

First, in a population game, consider two populations of agents going from $o$ to $d$. Suppose that a proportion $x^1_1$ of population $1$ is taking arc $1$ while the rest of the population ($x^1_2 = 1 - x^1_1$) uses arc 2, and similarly for population 2.

Next, in the splittable case, consider two players who  both have a stock to send from $o$ to $d$ and can divide their stock. Suppose that player $1$ sends a fraction $x^1_1$ of his stock by arc $1$ and the remaining   ($x^1_2$)  by arc $2$, and similarly for player $2$.
Finally, in the non splittable case, still consider two players having to send their stock from $o$ to $d$. However, they can no longer divide their stock, but have to send it entirely by one arc. Suppose that with probability $x^1_1$ player 1  sends it  by arc 1 and with probability $x^1_2 $  he sends it by arc 2, and similarly for player 2.

In all cases, the basic variable is $x \in X$, defined by $(x^1_1, x^2_1)\in [0,1]^2$.\\

Assume, more specifically, that the two participants are of size $\frac{1}{2}$ each and that the congestion is  $m$ on arc 1 and 1 on arc 2,  if the quantity of users is $m\in [0,1]$. Then it is easy to check that the equilibria are given by:

\noindent -- in framework $\mathbf{I}$, $ x^1_1 = x^2_1= 1$: all the agents choose arc 1;

\noindent  -- in framework $\mathbf{I\!I}$, $ x^1_1 = x^2_1= \frac{2}{3}$: both players send $\frac{2}{3}$ of their stock on arc 1 where $m= \frac{2}{3}$;

\noindent -- in framework $\mathbf{I\!I\!I}$,  the set of equilibria has one player choosing arc 1 (say $x^1_1 =1$) and the other choosing at random  with any $x^2_1 \in [0,1]$.
  
\section{Description of the  models} 

\subsection{Framework $\mathbf{I}$: population games} $ $\\
We consider here the nonatomic framework where each participant $i \in I$ corresponds to a  population of nonatomic agents.

The payoff (fitness) is defined by a family of continuous functions $\{F^{i}_p,\, i \!\in\! I, p\!\in\! S^i\}$ from $X$ to $\rit$, where $F^i_p(x)$ is the outcome of an agent  in population $i$ choosing $p$, when the environment is given by the basic variable $x$. 

An {\it equilibrium} is a point $x \in X$ satisfying:
\begin{equation}\label{Wardropdef}
 x^{i}_p > 0 \Rightarrow   F^{i}_p (x) \geq F^{i}_q (x), \quad \forall p,q \in S^i, \: \forall i\in I.
\end{equation} 
This corresponds to a {\it Wardrop equilibrium} \cite{War52}.

\begin{pro}[Smith \cite{Smi79},  Dafermos \cite{Daf80a}]\label{prop:vip_wardrop}$ $ \\
An equivalent characterization of (\ref{Wardropdef})  is through the variational inequality:
\begin{equation}\label{VIP_Wardrop_path}
\langle F^i(x), x^i - y^i \rangle \geq 0, \quad\forall  y^i \in X^i,  \forall i\in I,
\end{equation}
or alternatively:
\begin{equation}
\langle F(x), x - y  \rangle = \sum_{i\in I} \langle F^i(x), x^i - y^i \rangle \geq 0, \quad  \forall  y  \in X.
\end{equation}

\end{pro}



A special class  of population games corresponds to \emph{games with external interaction} where  each $F^i$ depends only on $x^{-i}$.

\subsection{Framework $\mathbf{I\!I}$: atomic splittable players} $ $\\
 In this case, each participant $i \in I$ corresponds to an atomic player with action set $X^i$. Given  functions $F^{i}_p$ as introduced above,  his gain is defined by:
\begin{equation*}
H^i(x) = \langle x^i, F^i(x)\rangle = \sum_{p\in S^i} x^{i}_p F^{i}_p(x).
\end{equation*}
In other words, it is the weighted average gain of all  fractions  $x^i_p$ allocated to different choices $p$.

An {\it equilibrium} is as usual a profile $x \in X$ satisfying:
\begin{equation}\label{Chucun_Nash_path}
H^i (x) \geq H^i (y^i, x^{-i}),  \quad  \forall  y^i \in X^i, \: \forall i\in I.
\end{equation}
Suppose that for all $p\in S^i$, $F^i_p(\cdot)$ is of class $\mathcal{C}^1$ on a neighborhood $\Omega$ of $X$, so that
\begin{equation*}
\frac{\partial \,H^i}{\partial \,x^{i}_{p}}(x)= F^i_{p}(x) + \sum_{q\in S^{i}}x^{i}_{q}\frac{\partial \,F^i_{q}}{\partial \,x^{i}_{p}}(x).
\end{equation*}
Let $\nabla^{i}\, H^{i}(x)$ stand for the gradient of $H^{i}(x^{i},\,x^{-i})$ with respect to $x^{i}$.
Then by a classical optimization criteria \cite{Kin86}  one has: 

\begin{pro}\label{prop:vip_atomicsplittable}
Any solution of (\ref{Chucun_Nash_path}) satisfies
\begin{equation}\label{VIP_Nash_path}
\langle \,\nabla {H}(x),\,x-y\rangle = \sum_{i\in I} \langle \,\nabla^i {H^i}(x),\,x^i-y^i\rangle \,\geq\,0,\qquad \forall\; y\in X.
\end{equation}
Moreover, if each  $H^i$ is concave with respect to $x^i$, there is equivalence.
\end{pro}

Variational inequalities characterizing Nash equilibrium in atomic splittable games (Haurie and Marcotte \cite{Hau85}) and those characterizing Wardrop equilibrium in nonatomic games have different origins. Inequalities in \eqref{VIP_Nash_path} for a Nash equilibrium are obtained as first order conditions, while inequalities in \eqref{VIP_Wardrop_path} for a Wardrop equilibrium are derived directly from its definition.

\subsection{Framework $\mathbf{I\!I\!I}$: atomic non splittable} $ $\\
We consider here an $I$-player game where the  payoff  is defined by a family of  functions $\{G^{i},\, i \!\in \! I\}$, all from $S = \prod_{i\in I} S^i$ to $\rit$. We still denote by $G$ the multilinear extension to $X$ where each $X^i = \Delta (S^i)$ is considered as the set of mixed actions. 

An equilibrium is a profile $x \in X$ satisfying:
\begin{equation}
G^{i} (x^i, x^{-i})  \geq  G^{i} (y^i, x^{-i}) ,  \quad  \forall  y^i \in X^i, \: \forall i\in I.
\end{equation}

Let $V\!G^i$ denote the vector payoff associated to $G^i$. Explicitly, $V\!G^{i}_p: X^{-i} \rightarrow \rit$ is defined by $V\!G^{i}_p (x^{-i}) = G^i(p, x^{-i})$, for all $p \in S^i$. Hence $G^i (x) = \langle x^i, V\!G^{i} (x^{-i}) \rangle$.

An equilibrium is thus  a profile $x\in X$ satisfying:
\begin{equation}\label{VIP_matrix}
\langle V\!G(x), x- y  \rangle = \sum_{i\in I} \langle V\!G^i(x^{-i}), x^i - y^i \rangle \geq 0, \quad  \forall  y  \in X.
\end{equation}
\subsection{Remarks}$ $ \\
Frameworks $\mathbf{I}$ and $\mathbf{I\!I\!I}$ have been extensively studied.
 Framework $\mathbf{I\!I\!I}$ corresponds to games with external interaction, but the multilinearity of $V\!G^i(x^{-i})$ will not be used in this paper.

Note that $F$, $\nabla H$ and $V\!G$ play similar roles in the three frameworks. This can be seen from the three variational characterizations of equilibrium: \eqref{VIP_Wardrop_path}, \eqref{VIP_Nash_path} and \eqref{VIP_matrix}. 

We call $F$, $\nabla H$ and $V\!G$ \emph{evaluation functions}  and denote them by  $\Phi$ in each of the three frameworks.

From now on, we consider the following class of games which includes (i) population games where $F^i_p$ are continuous on $X$ for all $i$ and $p$, (ii) atomic splittable games where $H^i$ is  concave and of class $\mathcal{C}^1$ on a neighborhood of $X^i$ for all $i$, and (iii) atomic non splittable games. A typical game in this class is denoted by $\Gamma(\Phi)$, where $\Phi$ is its evaluation function.

\begin{defi} 
 $N\!E(\Phi)$ is the set of  $x\in X$ satisfying:
 \begin{equation}\label{basic}
 \langle  \Phi(x), x - y  \rangle \geq 0,  \qquad \forall y \in X.
 \end{equation}
$N\!E(\Phi)$ is the set of equilibria of $\Gamma (\Phi)$.
 \end{defi}
The next result recalls general properties of  a variational inequality on a closed convex set. 
\begin{theo}
 Let $C\subset \rit^d$ be a closed convex set  and  $\Psi$ a map from $C$ to $\rit^d$. \\
 Consider the variational inequality: 
\begin{equation}
\langle \Psi(x), x-y \rangle \geq 0, \qquad \forall y\in C.
\end{equation}
Four equivalent  representations are given by:
\begin{equation}
\Psi(x) \in N_{C} (x),
\end{equation}
where $N_C (x)$ is the normal c\^one to $C$ at $x$;
\begin{equation}
\Psi(x) \in [T_{C} (x)]^{\bot},
\end{equation}
where $T_C(x)$ is the tangent c\^one to $C$ at $x$ and $[T_{C} (x)]^{\bot}$ its polar;
\begin{equation}\label{TC}
\Pi_{T_{C} (x)}\Psi(x)= 0,
\end{equation}
where $\Pi$ is the projection operator on a closed convex subset; and
\begin{equation}\label{Proj}
\Pi_C [x + \Psi (x)] = x.
\end{equation}
\end{theo}
\begin{proof}
$\langle \Psi(x), x-y \rangle \geq 0$ for all $y\in C$ is equivalent to $\Psi(x) \in N_C(x)$. Hence, $\Psi(x) \in [T_{C}(x)]^{\bot}$ and $\Pi_{T_{C} (x)}\Psi(x)= 0$ by Moreau's decomposition \cite{Mo}.

Finally the characterization of the projection gives:
\begin{equation*}
\big\langle x  + \Psi (x) - \Pi_C [x + \Psi (x)] , y - \Pi_C [x  + \Psi (x)] \big\rangle  \leq 0,  \qquad \forall y\in X.
\end{equation*}
Therefore, $\Pi_C [x + \Psi (x)] = x$ is  the solution. 
\end{proof}
Note that this result holds in a Hilbert space.


\section{Potential  and dissipative games}
\subsection{Potential games}
\begin{defi}
A real function $W$, of class $\mathcal{C}^1$ on a neighborhood $\Omega$ of $X$,  is a \emph{potential} for $\Phi$ if for each $i\in I$, there is a strictly positive function $\mu^i(x)$ defined on $X$ such that
 \begin{equation}\label{condition:pot}
\big\langle \nabla^{i} W(x) - \mu^i(x) \Phi^{i}(x) , y^i \big\rangle = 0, \quad \forall x \in X,   \forall y^i \in X^i_0, \, \forall i\in I,
\end{equation}
where $X^i_0 = \{y\in \rit^{|S^i|},\;  \sum_{p\in S^i} y_p  = 0 \}$ is the tangent space to $X^i$ and $\nabla^{i} $ denotes the gradient w.r.t. $x^i$. 

The game $\Gamma(\Phi)$ is then called a {\em potential game} and one says that $\Phi$  {\em derives}  from $W$.
\end{defi}

Some alternative  definitions of potential games have been used such as: 
\begin{equation}\label{condition:1}
\frac{\partial W(x)}{\partial x^i_p} - \frac{\partial W(x)}{\partial x^i_q} = \mu^i(x) [\Phi^i_p(x) -  \Phi^i_q(x)], \; \forall x\in \Omega,\,\forall p,q \in S^i
\end{equation}
or
 \begin{equation}\label{condition:2}
\frac{\partial W(x)}{\partial x^i_p}=\mu^i(x)  \Phi^i_p(x), \; \forall x \in \Omega, \,\forall p \in S^i.
\end{equation}
Remark that   \eqref{condition:2} yields \eqref{condition:1}, and 
\eqref{condition:1}  implies that the vector $\{ \frac{\partial W(x)}{\partial x^i_p} -  \mu^i  \Phi^i_p(x) \}_ {p \in S^i}$ is proportional to $(1, \ldots, 1)$, hence is orthogonal to 
$ X^i_0$, thus  (\ref{condition:pot}) holds.\\

Sandholm \cite{San00} defines a population potential game by \eqref{condition:2} with $\mu_i \equiv 1$ for all $i$. 

Monderer and Shapley \cite{Mon96} define potential games for finite games, which is equivalent to our definition \eqref{condition:1}  in framework $\mathbf{I\!I\!I}$.  

\begin{pro}\label{thm:potential_nonatom}
Let $\Gamma(\Phi)$ be a game with potential $W$.\\
\noindent 1. Every local maximum of $W$ is an equilibrium of $\Gamma(\Phi)$. \\
\noindent 2. If $W$ is concave on $X$, then any equilibrium of $\Gamma(\Phi)$ is a global maximum of $W$ on $X$.
\end{pro}
\begin{proof} 
Since  a local maximum $x$ of $W$  on the convex set $X$ satisfies:
\begin{equation}\label{conv}
\langle \nabla W (x), x - y  \rangle  \geq 0, \quad \forall y \in X,
\end{equation}
it follows from \eqref{condition:pot} that $\langle \mu^i(x) \Phi^i(x) , x^i\!  - \! y^i  \rangle \!  \geq \! 0$ for all $i$ and all $y\! \in\!  X$. This further yields \eqref{basic}. On the other hand, if $W$ is concave on $X$, a solution $x$ of \eqref{conv} is a global maximum of $W$ on $X$.
\end{proof}

\subsection{Dissipative games}
\begin{defi} 
The game $\Gamma(\Phi)$ is \emph{dissipative} if $\Phi$ satisfies:
\begin{equation*}
\langle \Phi (x) - \Phi(y), x - y  \rangle \leq 0,\qquad \forall\;(x,y) \in X\times X.
\end{equation*}
It is \emph{strictly dissipative} if
\begin{equation*}
\langle \Phi (x) - \Phi(y), x - y  \rangle < 0,\qquad \forall\;(x,y) \in X\times X \,\text{ with }\;x \neq y.
\end{equation*}
\end{defi}

In the framework of population games, Hofbauer and Sandholm \cite{HofSan09}  introduce this class of games and call them ``stable games''.

Notice that if $\Phi$  is dissipative and derives from  a  potential $W$, then $W$ is concave.

The set of equilibria in dissipative games has a specific structure (see  Hofbauer and Sandholm \cite[Proposition 3.1, Theorem 3.2]{HofSan09}) described as follows.
\begin{defi} 
 $S\!N\!E(\Phi)$ is the set of  $x\in X$ satisfying:
 \begin{equation}\label{eq:SNF}
 \langle  \Phi(y), x - y  \rangle \geq 0,  \qquad \forall y \in X.
 \end{equation}
 \end{defi}
 
 
\begin{pro}\label{lm:sne} $ $ \\
\begin{equation*}
S\!N\!E(\Phi) \subset N\!E(\Phi).
\end{equation*}
If $\Gamma(\Phi)$ is dissipative, then
\begin{equation*}
S\!N\!E(\Phi) = N\!E(\Phi).
\end{equation*}
\end{pro}

%
 
\begin{cor} 
If $\Phi$ is dissipative, $N\!E(\Phi)$ is convex.\\
A strictly dissipative game $\Gamma(\Phi)$ has a unique equilibrium.
\end{cor}

 The description is more precise in the smooth case. 
\begin{pro}
Suppose that $\Phi$ is of class ${\mathcal{C}}^1$ on a neighborhood $\Omega$ of $X$. Denote by $J_{\Phi}(x)$ the Jacobian matrix of $\Phi$ at $x$, i.e. $J_\Phi (x) = \big(\big(\frac{\partial  \Phi^i_p}{\partial x^j_q} \big)_{q\in S^j} \big)_{p\in S^i}$. Then, $\Phi$  dissipative implies that $J_\Phi(x)$ is negative semidefinite on $T_X(x)$, the tangent c\^one to $X$ at $x$. 
\end{pro}
\begin{proof}
Given $x\in X$ and $z\in T_X(x)$, there exists $\epsilon>0$ such that $x+tz\in X$ for all $t\in ]0,\epsilon]$. Hence $\langle \Phi(x+tz)-\Phi(x), x+tz-x \rangle \leq 0$, which implies that $\langle \frac{\Phi(x+tz)-\Phi(x)}{x}, z \rangle \leq 0$. Since $\Phi$ is of class $\mathcal{C}^1$, letting $t$ go to 0 yields $\langle J_{\Phi}(x)z, z\rangle \leq 0$.
\end{proof}

\begin{defi} 
Suppose that $\Phi$ is of class ${\mathcal{C}}^1$ on a neighborhood $\Omega$ of $X$. The game $\Gamma(\Phi)$ is {\em strongly dissipative} if $J_\Phi(x)$ is negative definite on $T_X(x)$.
\end{defi}
\section{Dynamics}
\subsection{Definitions}$ $\\
The general form of a dynamics  describing the evolution of the strategic interaction in game $\Gamma(\Phi)$ is given by: 
\begin{equation*}
\dot{x}=\cB_\Phi(x), \quad x\in X,
\end{equation*}
where $X$ is invariant,  so that  for each $i \in  I$, $\cB_\Phi^i (x) \in X^i_0$.
\medskip

First recall the definitions of several dynamics expressed in terms of $\Phi$.
\medskip

\noindent (1) \textit{Replicator dynamics} (RD) (Taylor and Jonker \cite{TayJon78})
 \begin{equation*}
 \dot x^{i}_p =  x^{i}_p [\Phi^{i}_p (x) - \overline{\Phi}^{i} (x)], \quad p\in S^i, i \in I,
 \end{equation*}
 where
 \begin{equation*}
\overline \Phi^{i} (x) = \langle x^i, \Phi^i(x) \rangle = \sum_{p\in S^i}  x^{i}_p \Phi^{i}_p (x)
\end{equation*}
  is the average evaluation for participant $i$. 
 \medskip

\noindent (2) \textit{Brown-von-Neumann-Nash dynamics} (BNN) (Brown and von Neumann \cite{BrownvonN1950}, Smith \cite{Smith1983b,Smi1984a}, Hofbauer \cite{Hofb2000})\\
\begin{equation*}    
 \dot x^{i}_p =  \hat{\Phi}^{i}_p -  x^{i}_p \sum_{q\in S^i}  \hat{\Phi}^{i}_q,  \quad p\in S^i, i \in I,
\end{equation*}
where $\hat{\Phi}^{i}_q = [\Phi^{i}_q (x) - \overline \Phi^i(x)]^+$ is called the ``excess evaluation'' of $p$. (Recall that $[t]^+ \triangleq \max \{t,0\}$.)
\medskip

\noindent (3) \textit{Smith dynamics} (Smith) (Smith \cite{Smi84})   
\begin{equation*}
\dot x^{i}_p =  \sum_{q\in S^i}   x^{i}_q [\Phi^{i}_p (x) -  \Phi^{i}_q (x)]^+  -   x^{i}_p  \sum_{q\in S^i} [\Phi^{i}_q (x) -  \Phi^{i}_p (x)]^+,  \quad p\in S^i,i \in I,
\end{equation*}
where $[\Phi^{i}_p (x) -  \Phi^{i}_q (x)]^+$ corresponds to pairwise comparison \cite{San11}.
 \medskip

\noindent (4) \textit{Local/direct projection dynamics} (LP) (Dupuis and Nagurney \cite{DuNagu93}, Lahkar and Sandholm \cite{LahkarSand2008})
   \begin{equation*}
 \dot x^{i} = \Pi_{T_{X^i}(x^i)} [\Phi^i(x)], \quad i \in I,
 \end{equation*}
where  $T_{X^i}(x^i)$ denotes  the tangent c\^one to $X^i$ at $x^i$.
\medskip

\noindent (5) \textit{Global/target projection dynamics} (GP) (Friesz et al. \cite{FrieszAl1994}, Tsakas and Voorneveld \cite{TsakasVoor2009})
 \begin{equation*}
 \dot x^{i} = \Pi_{X^i} [x^i  + \Phi^i(x)]  - x^{i},  \quad i \in I.
 \end{equation*}
 Recall that the two dynamics above are linked by:  $\Pi_{T_{X^i}(x^i)}[\Phi^i(x)]=\lim_{\delta\rightarrow 0+}\frac{\Pi_{X^i}(x^i+\delta \Phi^i(x))-x^i}{\delta}$.
 \medskip

\noindent (6) \textit{Best reply dynamics} (BR) (Gilboa and Matsui \cite{gm91})
 \begin{equation*}
 \dot x^{i} \in  B\!R^{i}(x)  - x^{i},  \quad i \in I,
 \end{equation*}
 where 
 \begin{equation*}
 B\!R^{i}(x) = \{ y^i  \in X^i,\;  \langle y^i - z^i , \Phi^i (x) \rangle \geq 0, \forall z^i \in X^i \}.
 \end{equation*}
 

\subsection{General properties}$ $\\
We define here properties expressed in terms of $\Phi$. 
\begin{defi}\label{def:PCNS}
Dynamics $\cB_\Phi$ satisfies:
\smallskip

\noindent i)  {\em positive correlation (PC)} (Sandholm \cite{San00}) if:
\begin{equation*}
\langle \cB_\Phi^i(x),  \Phi^i(x) \rangle >0, \quad  \forall  i \in I, \forall \,x \in X \text{ s.t. } \cB_\Phi^i(x)\neq 0.
\end{equation*}
(This  corresponds to MAD (myopic adjustment dynamics) (Swinkels \cite{Swin1993}): assuming the configuration given,  an unilateral change should increase the evaluation);
\smallskip

\noindent ii) {\em Nash stationarity} if:

for $x\in X$, $\cB_\Phi(x)=0$ if and only if $x$ is an equilibrium  of $\Gamma(\Phi)$.
\end{defi}

The next proposition collect results that have been obtained in different frameworks. We provide a unified treatment with simple and short proofs. 
\begin{pro}\label{prop:PC}
(RD), (BNN), (Smith), (LP),  (GP)  and (BR) satisfy (PC).
\end{pro} 
\begin{proof} $ $\\
\noindent (1) RD  (Sandholm \cite{San09,San11}):
\begin{align*}      
\langle \cB_\Phi^i(x), \Phi^i(x) \rangle 
& = \sum_{p\in S^i}  x^i_{p} \big[\Phi^i_p(x) - \overline{\Phi}^i(x)\big] \Phi^i_p(x) \\
& = \sum_{p\in S^i}  x^i_{p} \big[\Phi^i_p(x) - \overline{\Phi}^i(x)\big]^2 + \sum_{p\in S^i}  x^i_{p}\big[\Phi^i_p(x)- \overline{\Phi}^i(x)\big]\, \overline{\Phi}^i(x)\\
& = \sum_{p\in S^i}  x^i_{p} \big[\Phi^i_p(x) - \overline{\Phi}^i(x) \big]^2 +\big[\sum_{p\in S^i}  x^i_{p} \Phi^i_p(x)- \overline{\Phi}^i(x)\big]\, \overline{\Phi}^i(x)\\
& = \sum_{p\in S^i}  x^i_{p}\big[\Phi^i_p(x) - \overline{\Phi}^i(x) \big]^2  \geq 0.
\end{align*}
The equality holds if and only if for all  $p\in S^i$, $x^i_{p}\, [\Phi^i(x)-\overline{\Phi}^i(x)]^2=0$ or, equivalently  $x^i_{p}\, [\Phi^i(x)-\overline{\Phi}^i(x)]=0$, hence  $\cB^i_\Phi(x)=0$.
\medskip

\noindent (2) BNN (Sandholm \cite{San00,San05,San11}, Hofbauer \cite{Hofb2000}):
\begin{align*}
\langle \cB_\Phi^i(x), \Phi^i(x) \rangle 
& = \sum_{p\in S^i} \big[\hat{\Phi}^i_p(x) - x^ i_{p}\sum_{q\in S^i}\hat{\Phi}^i_{q}(x)\big] \Phi^ i_p(x)  
 = \sum_{p\in S^i} \hat{\Phi}^i_p(x)  \Phi^i_p(x) - \sum_{p\in S^i}x^i_{p}  \Phi^i_p(x)  \sum_{q\in S^i}\hat{\Phi}^i_{q}(x)\\
& = \sum_{p\in S^i} \hat{\Phi}^i_p(x)  \Phi^i_p(x) - \sum_{q\in S^i}\hat{\Phi}^i_{q}(x) \overline{\Phi}^i(x)
 = \sum_{p\in S^i} \hat{\Phi}^i_p(x) \big[\Phi^i_p(x) - \overline{\Phi}^i(x) \big]\\
& =\sum_{p\in S^i} (\hat{\Phi}^i_p(x))^2
 \geq  0.
\end{align*}
The equality holds if and only if for all $p\in S^i$, $\hat{\Phi}^i_p(x)=0$, 
 in which case $\cB^i_\Phi(x)=0$.
\medskip

\noindent (3) Smith (Sandholm \cite{San09,San11}):
\begin{align*}
\langle \cB_\Phi^i(x) ,   \Phi^i(x) \rangle 
& = \sum_{p\in S^i}\big(\sum_{q\in S^i} x^i_q [\Phi^i_p(x)- \Phi^i_q(x)]^+\big) \Phi^i_p(x) - \sum_{p\in S^i}x^i_p  \Phi^i_p(x) \sum_{q\in S^i} [\Phi^i_q(x)- \Phi^i_p(x)]^+ \\
& = \sum_{p,q} x^i_q  \Phi^i_p(x) [\Phi^i_p(x)- \Phi^i_q(x)]^+ -  \sum_{q,p}x^i_q  \Phi^i_q(x) [\Phi^i_p(x)-\Phi^ i_q(x)]^+ \\
&  = \sum_{p,q} x^i_q ([\Phi^i_p(x)- \Phi^i_q(x)]^+)^2
 \geq 0.
\end{align*}
The equality holds if and only if for all $q \in S^i$, either $x^i_q=0$ or $\Phi^i_{q}(x)\geq  \Phi^i_{p}(x)$ for all $p\in S^i$,
 in which case $\cB^i_\Phi(x)=0$. 
\medskip

\noindent (4) LP (Lahkar and Sandholm \cite{LahkarSand2008}, Sandholm \cite{San11}): Recall that, if $N_{X^i}(x^i)$ denotes the normal c\^one to $X^i$ at $x^i$  (the polar of  $T_{X^i}(x^i)$), then for any $v \in \rit^{S^i}$:
 $v=\Pi_{T_{X^i}(x^i)}v+\Pi_{N_{X^i}(x^i)}v$ and $\langle \Pi_{T_{X^i}(x^i)}v,  \Pi_{N_{X^i}(x^i)}v\rangle=0$ (Moreau's decomposition, \cite{Mo}).
 Thus,
\begin{align*}
\langle \cB_\Phi^i(x) ,   \Phi^i(x) \rangle 
&  =\big\langle \Pi_{T_{X^i}(x^i)} [\Phi^i(x)],  \Phi^i(x) \big\rangle 
 = \big\langle \Pi_{T_{X^i} (x^i)} [\Phi^i(x)], \Pi_{T_{X^i}(x^i)} [\Phi^i(x)]+\Pi_{N_{X^i}(x^i)} [\Phi^i(x)] \big\rangle\\
& = \big\| \Pi_{T_{X^i}(x^i)} [\Phi^i(x)] \big\|^2 \geq 0,
\end{align*}
and the equality holds if and only if $\Pi_{T_{X^i}(x^i)} [\Phi^i(x)]=0$, i.e. $\cB^i_\Phi(x)=0$.
\medskip

\noindent (5) GP (Tsakas and Voorneveld \cite{TsakasVoor2009}): Let $z=x+\Phi (x)$. Then,
\begin{align*}
\langle \cB_\Phi^i(x) ,   \Phi^i (x) \rangle 
& = \big\langle \Pi_{X^i}(z^i)-x^i, z^i-x^i \big\rangle 
 = \big\langle \Pi_{X^i}(z^i)-x^i, z^i- \Pi_{X^i}(z^i) + \Pi_{X^i}(z^i) -x^i \big\rangle\\ 
& = - \big\langle x^i- \Pi_{X^i}(z^i), z^i- \Pi_{X^i}(z^i) \big\rangle + \| \Pi_{X^i}(z^i)-x^i\|^2\\
& \geq   \| \Pi_{X^i}(z^i)-x^i\|^2 
 \geq 0.
\end{align*}
The second last inequality holds since  $x^i\in X^i$,  thus $\big\langle x^i- \Pi_{X^i}(z^i), z^i- \Pi_{X^i}(z^i) \big\rangle \leq 0$. The equality occurs in both inequalities if and only if $x^i=\Pi_{X^i}(z^i)$, in which case $\cB^i_\Phi(x)=0$.
\medskip

\noindent (6) BR (Sandholm \cite{San11}):
\begin{align*}
\langle \cB_\Phi^i(x) ,   \Phi^i (x) \rangle = \langle  y^i  - x^i  ,   \Phi^i (x) \rangle \geq 0
\end{align*}
since $y^i \in B\!R^i (x)$. The equality holds if and only if $x^i \in B\!R^i (x)$, hence $\cB^i_\Phi(x)=0$.
\end{proof}
\smallskip

\begin{pro}\label{prop:NS}
(BNN), (Smith), (LP), (GP)  and (BR) satisfy Nash stationarity on $X$. \\
(RD) satisfy Nash stationarity on $\inte X$.
\end{pro}
\begin{proof} $ $ \\
\noindent (1) BNN (Sandholm \cite{San00,San05,San11}): $x\in N\!E(\Phi)$ is equivalent to:  $\hat{\Phi}_{p}(x) =0$ for all $p\in S$. Hence $\cB_\Phi(x)=0$.\\
Reciprocally, assume the existence of $i\in I$ and $p\in S^i$ such that  $\hat{\Phi}^i_{p}(x) >0$. Since there exists $q$ with $x^i_q >0$ and $\hat{\Phi}^i_{q}(x) =0$,  one obtains $\cB^i_{\Phi,q}(x)<0$, contradiction.
 \medskip

\noindent (2) Smith (Sandholm \cite{San09,San11}): Assume $x\in N\!E(\Phi)$. Then for all $i\in I$ and $q \in S^i$, either $x^i_q=0$ or $\Phi^i_q(x)\geq  \Phi^i_{p}(x)$ for all $p\in S^i$. Hence $\cB_\Phi(x)=0$.

Reciprocally, assume the existence of $i\in I$ and $p\in S^i$ such that  $x^i_p >0$ and $[\Phi^i_q (x) - \Phi^i_p (x)]^+ >0$. Choose such a $p$ with smallest $\Phi^i_p (x)$ then one obtains $\cB^i_{\Phi,p}(x)<0$.
\medskip

\noindent (3) LP (Lahkar and Sandholm \cite{LahkarSand2008}, Sandholm \cite{San11}): The result follows from \eqref{TC}.
 \medskip
 
\noindent (4) GP (Tsakas and Voorneveld \cite{TsakasVoor2009}): The result follows from \eqref{Proj}.
\medskip

\noindent (5) BR (Sandholm \cite{San11}): By definition $\cB_\Phi(x)=0$ if and only if $x^i \in B\!R^i (x)$ hence $x \in N\!E (\Phi)$.
\medskip

\noindent (6) RD (Sandholm \cite{San09,San11}): Assume $x\in \inte X \cap N\!E(\Phi)$. Then, $\Phi^i_p(x)=\overline{\Phi}^i(x)$ for all $i \in I$ and $p\in S^i$ and thus $\cB_\Phi(x)= 0$.

Reciprocally, if $x\in \inte X$ and $\cB_\Phi(x)=0$, $x$ is equalizing hence in $N\!E(\Phi)$.  
\end{proof}

\subsection{Potential games}$ $\\
We establish here results that are valid for all three frameworks of games.

\begin{pro} \label{prop:pot_lya}
Consider a  potential game $\Gamma(\Phi)$ with potential function $W$. If the dynamics $\dot{x}=\cB_\Phi(x)$  satisfies 
 (PC), then $W$ is a strict  Lyapunov function for $\cB_\Phi$. Besides, all $\omega$-limit points are rest points of $\cB_\Phi$. 
\end{pro}
\begin{proof}
Consider $x\in X$. Let $\{x_t\}_{t\geq 0}$ be the trajectory of $\cB_\Phi$ with initial point $x_0 = x$, and $V_t = W(x_t)$ for $t\geq 0$. Then
\begin{equation*}
 \dot V_t = \langle \nabla W(x_t),  \dot{x}_t \rangle =  \sum_{i\in I}  \langle \nabla^i W(x_t),  \dot{x}^i_t \rangle = \sum_{i\in I} \mu^i(x) \langle  \Phi^i(x_t), \dot{x}^i_t \rangle   \geq 0.
\end{equation*}
(Recall that $\dot{x}_t \in X_0$.) Moreover, $\langle \Phi^i(x_t), \dot{x}^i_t \rangle = 0$ holds for all $i$ if and only if $\dot{x}=\cB_\Phi(x_t) = 0$.

One concludes by using Lyapunov's theorem (e.g.  \cite[Theorem 2.6.1]{HSy}).
\end{proof}
  
This result is proved by Sandholm \cite{San00} for the version of potential games in framework $\mathbf{I}$ defined by \eqref{condition:2}. 
\bigskip

It follows that, with the appropriate definitions,  the convergence results established for several dynamics and  potential games  in framework $\mathbf{I}$  or $\mathbf{III}$ extend to all dynamics and frameworks. Explicitly:
\begin{pro} \label{prop:potential_lyapunov}
Consider a potential game $\Gamma(\Phi)$ with potential function $W$.\\
 If the dynamics is (RD), (BNN), (Smith), (LP),  (GP) or (BR),  $W$ is a strict  Lyapunov function for $\cB_\Phi$. \\
In addition, except for (RD),   all $\omega$-limit points are equilibria of $\Gamma(\Phi)$. 
\end{pro}
\medskip

\subsection{Dissipative games}      

We apply also for this class the previous ``dictionary" used for potential games and dynamics.

\begin{pro} \label{prop:dissip_lya}
Consider  a dissipative game $\Gamma(\Phi)$.
\smallskip


\noindent (1) RD: Let $x^*\in N\!E(\Phi)$. Define \cite{HofSan09}:
 \[H(x)= \sum_{i \in I}\sum_{p\in \textrm{supp}(x^{i*})}x^{i*}_p \, \ln \frac{x^{i*}_p}{x^{i}_p}. \]
Then $H$ is a local Lyapunov function.\\
If $\Gamma(\Phi)$ is strictly dissipative, then $H$ is a local strict Lyapunov function. 
\medskip

\noindent (2) BNN: Assume  $\Phi$  $\mathcal{C}^1$ on a neighborhood $\Omega$ of $X$. Define \cite{Smith1983b,Smi1984a,Hofb2000}:
\[H(x) = \frac{1}{2}\sum_{i \in I}\sum_{p\in S^{i}}\hat{\Phi}^{i}_p(x)^{2}. \]
Then $H$ is a strict Lyapunov function which is minimal on $N\!E(\Phi)$.
\medskip

\noindent (3) Smith: Assume  $\Phi$  $\mathcal{C}^1$ on a neighborhood $\Omega$ of $X$. Define  \cite{Smi84}:
\[H(x) = \sum_{i \in I}\sum_{p,q\in S^{i}} x^{i}_p \big\{[\Phi^{i}_q(x)-\Phi^{i}_p(x)]^{+}\big\}^{2}. \]
Then $H$ is a strict Lyapunov function which is minimal on $N\!E(\Phi)$.
\medskip

\noindent (4) LP: Let $x^*\in N\!E(\Phi)$. Define \cite{NaguZhang97, ZhangNa1997,PappaPassa2002}:
\[H(x) = \frac{1}{2} \|x-x^*\|^2. \]
Then $H$ is a Lyapunov function.\\
If $\Gamma(\Phi)$ is strictly dissipative, then $H$ is a strict Lyapunov function. 
\medskip

\noindent (5) GP: Assume  $\Phi$  $\mathcal{C}^1$ on a neighborhood $\Omega$ of $X$. Define \cite{PappaPassa2004}:
\[H(x) = \sup_{y\in X} \langle\, y-x, \Phi(x) \,\rangle - \frac{1}{2}\|y-x\|^2. \]
Then $H$ is a Lyapunov function.\\
If $\Gamma(\Phi)$ is strictly dissipative, then $H$ is a strict Lyapunov function. 
\medskip

\noindent (6) BR:  Assume  $\Phi$  $\mathcal{C}^1$ on a neighborhood $\Omega$ of $X$. Define \cite{HofSan09}:
\[H(x) = \sup_{y\in X} \langle\, y-x, \Phi(x) \,\rangle. \]
Then $H$ is a strict Lyapunov function which is minimal on $N\!E(\Phi)$.
\end{pro}
The proof is in Apprendix.
\section{Example: congestion games}
An eminent example of the games studied in this paper is  a network congestion game, or routing game. The underlying network is a finite directed graph $G=(V,A)$,  where $V$ is the set of nodes and   $A$ the set of links. The vector $\bol=(l_a)_{a\in A}$ denotes a family of cost functions from $\rit$  to $\rit^+$: if the aggregate weight on arc $a$ is $m$, the cost per unit (of weight) is $l_{a}(m)$. 

The set $I$ of participants is finite. A participant $i$ is characterized by his {\em weight} $m^{i}$ and an {\em origin/destination pair} $(o^{i}, d^{i})\in V \times V$ such that the constraint is to send a quantity $m^{i}$ from $o^{i}$ to $d^{i}$. The set of choices of participant  $i\in I$ is $S^{i}$: directed acyclic paths linking  $o^{i}$  to $d^{i}$ and available to $i$. Let $P = \cup_{i\in I} S^i$.

Assume that, for all arcs $a \in A$, the  function $l_{a}$ is continuous and finite on a neighborhood $U$ of the  interval $[0,\,M]$ and positive on $U \cap \mathbb{R}_{+}$, where $M=\sum_{i\in I}m^{i}$ is the aggregate weight of the players.

In each of the three frameworks considered in this paper, a participant is respectively a population of nonatomic agents ($\mathbf{I}$), an atomic splittable player ($\mathbf{I\!I}$)  and an atomic non splittable player ($\mathbf{I\!I\!I}$). Thus, in framework $\mathbf{I}$, a fraction $x^i_p$ of population $i$ takes path $p$; in framework $\mathbf{I\!I}$, $x^{i}_{p}$ is the proportion  of the weight $m^i$ sent on path $p$ by player $i$; in framework $\mathbf{I\!I\!I}$, $x^i_p$ is the probability with which player $i$ take path $p$. The basic variable $x$ is 
a profile of strategies of the participants and the corresponding set is $X = \prod_{i\in I} X^i$. 

In frameworks $\mathbf{I}$ and $\mathbf{I\!I}$, a strategy  $x^i$ induces a {\em flow $f^i$ on the arcs} (or simply {\em flow}) for each  participant $i$. Explicitly, the weight on arc $a$ 
is $f^{i}_{a}=\sum_{p\in S^i,\, p\ni a}m^i x^{i}_{p}$. Define the aggregate configuration $\bz=(\bz_p)_{p\in P}$, where $\bz_p=\sum_{i\in I, S^i \ni p} m^i x^i_p$. The {\it aggregate flow} is $\bof=(\bof_{a})_{a\in A},$ with $\bof_{a} = \sum_{i\in I} f^{i}_{a}$ the aggregate weight on arc $a$. Notice that $\bof$ can also be induced by the aggregate configuration $\bz$. Denote $f^{-i}_a=\bof_a-f^i_a$.

Given an aggregate configuration $\bz$  and aggregate flow $\bof$, the {\em vector of congestion on the arcs} is $\bol(\bof)=\{l_{a}(\bof_{a})\}_{a\in A}$. This specifies now the {\it cost of a path} $p$ by 
$c_{p}(\bz)=\sum_{a\in p}l_{a}(\bof_{a})$. The corresponding vectors are $c^i (\bz) = (c_{p}(\bz))_{p\in S^i}$ for $i\in I$, $c(\bz)=(c^i(\bz))_{i\in I}$, and $l^i(\bof)=\bol(\bof)$ for all $i\in I$. In particular, path costs $c_p$ and arc costs $l_a$ are determined only by the {\emph aggregate} configuration or the {aggregate} flow.

The evaluation functions in the first two frameworks are respectively:
\begin{description}
\item[$\mathbf{I}$] population: $\Phi^i_p(x)=-c_p(\bz)$.
\item[$\mathbf{I\!I}$] atomic splittable: $\Phi^i_p(x)= - \frac{\partial u^i(x)}{\partial x^i_p}$, where $u^i(x)$ is the cost to atomic player $i$:
\begin{equation*}
 u^{i}(x) =\langle x^{i},\, c^i(\bz) \rangle =\sum_{p \in S^{i}} x^{i}_{p} \,c_{p} (\bz) = \frac{1}{m^i} \langle f^{i},\, \bol(\bof) \rangle = \frac{1}{m^i}\sum_{a\in A} f^{i}_{a}\,l_{a}(\bof_{a}).
\end{equation*}
\end{description}

In framework $\mathbf{I\!I\!I}$, first consider the arc flow $f$ and aggregate arc flow $\mathbf{f}$ induced by a pure-strategy profile $s$: $\mathbf{f}_a(s)=\sum_{i\in I}f^i_a(s)$ and $f^i_a(s)=m^i \ind_{a \in s^i}$. Then 
 the evaluation function is $\Phi^i_p(x)=-V\!U^i_p(p,x^{-i})=-U^i(p,x^{-i})$, where
\begin{equation*}
U^i(x)=  \sum_{s \in \prod_{j\in I} S^j}\prod_{j\in I}x^j_{s^j} \sum_{a\in s^i}l_a(\mathbf{f}_a(s))
\end{equation*}
\smallskip

Congestion games  are thus natural  settings where each kind of participants occurs. However one can even consider a game where   participants  of different natures coexist:  some of them being of category $\mathbf{I}$,  $\mathbf{I\!I}$ or $\mathbf{I\!I\!I}$. This leads to the notion of composite game which is introduced in full generality in Section 6. Consider here a simple example where there are three participants in the network: population 1 of nonatomic agents of weight $m^1$, atomic splittable player 2 of weight $m^2$, and atomic non splittable player $3$ of weight $m^3$. Suppose that their basic variables are $x^1=(x^1_p)_{p\in S^1}$, $x^2=(x^2_q)_{q\in S^2}$ and $x^3=(x^3_s)_{s\in S^3}$ respectively. Here $x^1$ and $x^2$ describe pure strategies while $x^3$ specifies a mixed strategy.

Let us first look at the pure-strategy profiles. Given a pure strategy $s\in S^3$,  let  $x=(x^1, x^2, s)$, $\bof(x)$ be the induced aggregate flow: $\bof_a(x) = \sum_{p\in S^1,\, p\ni a}m^1 x^{1}_{p} + \sum_{q\in S^2,\, q \ni a}m^2 x^{2}_{q} +  m^3  \ind_{\{a \in s \}}$ and $\bz$  the induced configuration. The corresponding cost of a path $p$ is $ c_p (\bz) = \sum_{a\in p}l_{a}(\bof_a(x))$. Therefore the cost to an agent in population 1 using path $p$ (if $x^1_p>0$) is $c_p(\bz)$ for all $p\in S^1$, the cost to  atomic splittable player 2 is $u^2(x) = \sum_{q\in S^2} x^2_q c_q(\bz)$, and the cost to atomic non splittable player 3 is $u^3(x)=c_s(\bz)$.

%
%
%

Consider now a strategy profile $x=(x^1,x^2,x^3)$.  This induces a distribution on pure strategy profiles $(x^1,x^2, s)$ hence  a distribution on the sets of aggregate distributions $\bz _s$. The cost to each player is then obtained by taking the relevant expectation. Explicitly, 
for population 1, the cost to an agent in the population using path $p$ (if $x^1_p>0$) is $\sum_{s\in S^3}x^3_s c_p(\bz_s)$, while its evaluation function is $\Phi^1(x)=(\Phi^1_p(x))_{p\in S^1}$ where $\Phi^1_p(x)=-\sum_{s\in S^3}x^3_s c_p(\bz_s)$. For atomic splittable player 2, his cost is $u^2(x)   = \sum_{s\in S^3}x^3_s  \sum_{q\in S^2} x^2_q c_q(\bz_s)$, while his evaluation function is $\Phi^2(x)=(\Phi^2_q(x))_{q\in S^2}$ where $\Phi^2_q(x)=-\frac{\partial u^2(x)}{\partial x^2_q}$. For atomic non splittable player 3, his cost is $u^3(x)=  \sum_{s \in  S^3}x^3_{s} c_s(\bz_s)$, while his evaluation function is $\Phi^3(x)=(\Phi^3_s(x))_{s\in S^3}$ where $\Phi^3_s(x)= -c_s(\bz_s)$.

Through this example, one can further see that congestion games are a natural example of an aggregative  game (see \cite{Se}) where the payoff of a participant $i$ depends only on $x^i \in X^i$ and  on some fixed dimensional   function $\alpha^i (\{x^j\}_{j \not=i}) \in\Delta (\rit^P)$ (here the aggregate distribution induced by the participants $-i$). Because of the aggregative property of congestion games, one can show that accumulation points of flows induced by Nash equilibria for a sequence of  composite congestion games, when the atomic players split into identical players with vanishing weights, are Wardrop equilibria of the `limit' nonatomic game, i.e. the nonatomic game  obtained where the atomic splittable players in the previous sequences games are replaced by populations of nonatomic agents. This result shows intrinsic link between the different frameworks discussed in this paper. The reader is referred to  Haurie and Marcotte~\cite{Hau85} or Wan \cite{Wan11} for details.
\medskip

In framework $\mathbf{I\!I}$, $u^i$ is of class $\mathcal{C}^1$ and convex when the arc cost functions satisfy a mild condition, as the following lemma shows \cite[Lemma 29]{Wan13}. 
\begin{pro}\label{u_convex_i}
In $\Gamma(\Phi)$, if  each  cost function $l_{a}$ is of class $\mathcal{C}^1$, nondecreasing and convex on $U$, for all arc $a\in A$, then $u^{i}(x^{i},\,x^{-i})$ is convex with respect to $x^{i}$ on a neighborhood of $X^{i}$ for all fixed $x^{-i}\in X^{-i}$.
\end{pro}

\section{Composite Games}
\subsection{Composite games and variational inequalities}$ $ \\
We have seen that the properties of equilibrium and dynamics in the three frameworks all depend on the evaluation function $\Phi$ and the variational inequalities associated to it. Based upon this idea, let us define a more general class of games called {\em composite games}, which exhibit different categories of players. Composite congestion games with  participants of categories $\mathbf{I}$ and $\mathbf{I\!I}$ have been studied by Harker \cite{Ha88}, Boulogne et al. \cite{Bou02}, Yang and Zhang \cite{YangZh2008} and Cominetti et al. \cite{Com09},  among others.

Consider a finite set $I_1$ of populations composed of nonatomic agents, a finite set $I_2$ of atomic splittable players and a finite set $I_3$ of atomic non splittable players. Let $I=I_1\cup I_2\cup I_3$. 

All the analysis of Sections 3 and 4 extend to this setting where $x = \{x^i\}_{i \in I_1\cup I_2\cup I_3}$ and $\Phi^i (x)$ depends upon the category of participant $i$. 

Explicitly, there are finitely many ``participants" $i\in I$ and each of them has finitely many ``choices" $p\in S^i$. Vector $x^i = \{x^{i}_p,\, p \!\in\! S^i\}$ belongs to simplex $X^i = \Delta(S^i)$ on $S^i$ and $X = \prod_{i\in I} X^i$.\\
For $i \in I_1$, the payoff $F^i_p, p \in S^i$ is a continuous function  on $X$ and $\Phi^i = F^i$.\\
For $i \in I_2$, $F^i_p$, $p \in S^i$ is a continuous function on $X$ and the payoff $H^i (x) = \langle x^i, F^i(x) \rangle$  is concave and of class $\mathcal{C}^1$ on a neighborhood of $X^i$. Then $\Phi^i = \nabla^i H^i$.\\
For $i  \in I_3$,  $V\!G^i_p$ is continuous on $X^{-i}$, the payoff is $G^i (x) = \langle x^i, V\!G^i (x^{-i}) \rangle$ and $\Phi^i = V\!G^i$.

Let this composite game be denoted by $\Gamma(\Phi)$.

The main point to note is that the evaluation by participant $i$ is independent of the category of his/her opponent  (their evaluation functions $\Phi ^{-i}$) hence  the analysis of Section 2 implies the following.

\begin{pro}\label{VIP-compo}
 $x\in X$ is a composite equilibrium of a composite game $\Gamma(\Phi)$ if and only if 
\begin{equation}
\langle \Phi(x), x-y \rangle \geq 0, \qquad \forall  y \in X.
\end{equation}
\end{pro}

An example of such a composite game is a congestion game with the three categories of participants in a network (see Section 5). 

\medskip

\subsection{Composite potential games and composite dynamics}$ $\\
Once the equilibria of a composite  game are formulated in terms of solutions of  variational inequalities, those properties of equilibria and dynamics based on such a formulation in the three frameworks discussed so far are naturally inherited in the composite setting.

The general form of a dynamics in a composite  game $\Gamma(\Phi)$ is again 
\begin{equation*}
\dot{x}=\cB_{\Phi}(x).
\end{equation*}
The definitions of positive correlation and Nash stationarity for $\cB_{\Phi}$ are exactly the same as in Definition \ref{def:PCNS}.

Arguments similar to those for Propositions \ref{prop:PC} and \ref{prop:NS} show the following.
\begin{pro}\label{prop:specific}
In composite  game $\Gamma(\Phi)$, composite dynamics (Smith), (BNN), (LP), (GP) and (RD) satisfy (PC) and Nash stationarity on $X$.\\
 (RD) satisfies (PC) and Nash stationarity on $\text{int} X$.
\end{pro}

The definition of a potential game and that of a dissipative game for composite games are analogous  to those for each of the three frameworks.  Note that the condition for each participant is independent of the others. Hence the corresponding properties of composite dynamics for these specific classes of composite games can be proved in the same way as for Propositions \ref{prop:potential_lyapunov} and \ref{prop:dissip_lya}.

\begin{defi}
A composite  game $\Gamma(\Phi)$ is a {\em composite potential game} if there is a real-valued function $W$ of class $\mathcal{C}^1$ defined on a neighborhood $\Omega$ of $X$, called  {\em potential function}, and strictly positive functions $\mu^i, i \in I$ on $X$  such that for all $x \in X$,
 \begin{equation}\label{cnd:composite_potential}
\big\langle \nabla^{i} W(x) - \mu^i (x) \Phi^{i}(x) , y^i \big\rangle = 0, \quad \forall x \in X,  \forall y^i \in X^i_0, \, \forall i\in I,
\end{equation}
where $X^i_0 = \{y\in \rit^{|S^i|},\; \sum_{p\in S^i} y_p  = 0 \}$ for all $i\in I$. 
\end{defi}

\begin{pro} \label{prop:pot_lya}
In a composite potential  game $\Gamma(\Phi)$ with potential function $W$, if  the dynamics $\dot{x}=\cB_{\Phi}(x)$ satisfies Nash stationarity and (PC), then $W$ is a global  strict Lyapunov function for $\cB_{\Phi}$. Besides, all $\omega$-limit points are rest points of $\cB_{\Phi}$. 
\end{pro}

\begin{cor}
A potential function $W$ of a composite potential  game $\Gamma(\Phi)$ is a global strict Lyapunov function for (RD),  (Smith), (BNN), (LP), (GP) and (BR).
\end{cor}

\begin{defi}
A composite game $\Gamma(\Phi)$ is a {\em composite dissipative game} if $-\Phi$ is a monotone operator on $X$. The game is {\em strictly dissipative} if $-\Phi$ is strictly monotone on $X$.
\end{defi}

\begin{pro} 
If a composite congestion game $\Gamma(\Phi)$ is dissipative, then one can find Lyapunov functions for (RD), (Smith), (BNN), (LP), (GP) and (BR) as defined in Proposition \ref{prop:dissip_lya} (with the assumption that $\Phi$ is of class $\mathcal{C}^1$ for (Smith), (BNN), (GP) and (BR)).
\end{pro}

As a matter of fact, we can obtain results stronger than Proposition 6.1, Corollary 6.1 and Proposition 6.3 which consider only homogeneous dynamical  models, i.e. where all the participants follow the same dynamics. Looking more closely into the proofs for the results in Sections 4.2-4.4,  one can see that the results extend naturally to heterogeneous dynamical models where the participants follow specific dynamics.  

Basically note that the condition for  positive correlation can be written component by component, hence it corresponds to a unilateral property: it depends only for each participant $i$ on the evaluation $\Phi^i$ and the dynamics $\cB^i_{\Phi}= \cB^i_{\Phi^i}$. 
For example, in a composite potential game, as long as  each of the dynamics followed by different participants satisfies (PC), the potential function is a strict Lyapunov function. Because of this unilateral property, the dynamics in this paper belong to the class of \emph{uncoupled dynamics} studied by Hart and Mas-Colell \cite{HM2003b}.
\medskip

\subsection{One example of a composite potential game}$ $ \\
Consider a composite congestion game, with three categories of participants $i \in  I = I_1 \cup I_2 \cup I_3$,   of weight $m^i$ each, taking place in a network composed of  two nodes $o$ and $d$ connected by a finite set $A$ of parallel arcs. 
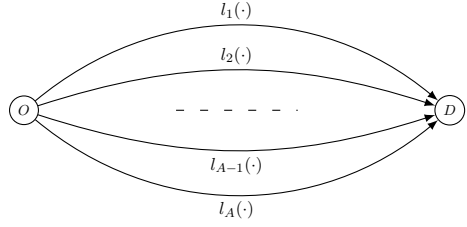
\begin{figure}[htbp!]
\caption{Example of a composite potential game\label{fg:exp}}
\begin{center}
\begin{tikzpicture}[scale=2]
\node[draw,circle,scale=0.5] (O)at(-1.4,0) {$O$};
\node[draw,circle,scale=0.5] (D)at(1.4,0) {$D$};
\node (P)[scale=0.5] at (-2,0) {};
\node (Q)[scale=0.5] at (2,0) {};
\draw[->,>=latex] (O) to[bend left=40] node[midway,above,scale=0.6]{$l_1(\cdot)$}(D);
\draw[->,>=latex] (O) to[bend left=18] node[midway,above,scale=0.6]{$l_2(\cdot)$} (D);
\draw[->,>=latex] (O) to[bend right=18]  node[midway,below,scale=0.6]{$l_{A-1}(\cdot)$}(D);
\draw[->,>=latex] (O) to[bend right=40]  node[midway,below,scale=0.6]{$l_A(\cdot)$}(D);
\draw [loosely dashed]  (-0.4,0) -- (0.4,0);
\end{tikzpicture}
\end{center}
\end{figure}

 Denote by $s = (s^k)_{k \in I_3}\in S_3=A^{I_3}$ a pure strategy profile of participants in $I_3$  and let $z=((x^i)_{i\in I_1}, (x^j)_{j\in I_2}, (s^k)_{k\in I_3})$. Let $\bof(z)$ be the aggregate flow induced by the pure-strategy profile $z$. Namely: $\bof_a(z) = \sum_{i \in I_1} m^i x^i_a  +  \sum_{j \in  I_2}  m^j x^j_a + \sum_{k \in I_3}  m^k  \ind_{\{s^k=a\}}$.
\begin{theo}
Assume that for all  $a\in A$, the per-unit cost function is affine, i.e. $l_a(m)=b_a m+d_a$, with $b_a>0$ and $d_a\geq 0$. Then a composite congestion game in this network is a potential game. \\
A potential function defined on $X$ is given by:
\begin{equation}\label{ex_pot}
W(x)=-\sum_{s \in S_3} \Big(\prod_{k \in I_3}x^{k}_{s^k}\Big) \Big\{\frac{1}{2}\sum_{a \in A} b_a \big[(\bof_a(z)^2+\sum_{j\in I_2}(m^j x^{j}_a)^2+\sum_{k\in I_3}(m^k)^2 \ind_{\{s^k=a\}} \big]+\sum_{a\in A}d_{a}\bof_a(z)\Big\},
\end{equation}
with  $\mu^i(x)\equiv m^i$ for all $i\in I=I_1\cup I_2 \cup I_3$ and all $x\in X$. 
\end{theo}
\begin{proof}
First notice that function $W$ defined in \eqref{ex_pot} is the multilinear extension of the following function defined on $Z$, the set of pure-strategy profiles:
\begin{align*}
W(z)=-\frac{1}{2}\sum_{a \in A}b_a \big[(\bof_a(z))^2+\sum_{j\in I_2}(m^j x^{j}_a)^2 +\sum_{k\in I_3}(m^k)^2 \ind_{\{s^k=a\}} \big]+\sum_{a\in A}d_a\bof_a(z).
\end{align*}
The per-unit cost to take arc $a$ when the pure-strategy profile is $z$ is $c_a(z)=  b_a \bof_a(z)+d_a$, for all arc $a\in A$. Recall that this is the opposite of the evaluation 
for the nonatomic players in population $i\in I_1$:  $\Phi^i_a (z) = - c_a(z)$. On the other hand,
\begin{equation}\label{eq:ex1}
- \frac{\partial W(z)}{\partial x^{i}_a} =\frac{1}{2} b_a \cdot 2 \bof_a(z)\cdot m^i +d_a m^i  = m^i \big[b_a \bof_a(z)+d_a \big]  = m^i c_a(z).
\end{equation}

For an atomic splittable player $j\in I_2$, when the pure-strategy profile is $z$, the  cost is $u^{j}(z)=  \sum_{a\in A} x^j_a \big[b_a\bof_a(z)+d_a\, \big]$. Therefore,
\begin{equation*}
 \frac{\partial u^j(z)}{\partial x^j_a}= b_a\bof_a(z)+d_a + b_a m^j x^{j}_a = -  \Phi^j_a (z) .
\end{equation*}
On the other hand,
\begin{equation}\label{eq:ex2}
- \frac{\partial W(z)}{\partial x^{j}_a} 
=\frac{1}{2} \,b_a \big[\, 2 m^j \bof_a(z) +2 (m^j)^2 x^{j}_a\,\big] +d_a m^j 
= m^j \big[b_a\bof_a(z) + b_a m^j x^{j}_a + d_a \big]
= m^j \frac{\partial u^j(z)}{\partial x^{j}_a}.
\end{equation}

Finally, for an atomic non splittable player $k\in I_3$, the cost to take arc $a$ when the other players play $z^{-k}$ is $u^k(z^{-k},a) = b_a\bof_a(z^{-k},a)+d_a = \Phi^k_a$. On the one hand,
\begin{equation*}
u^k(z^{-k},a) - u^k(z^{-k},r) =[\, b_a \bof_a(z^{-k},a)+d_a \,] - [\, b_r \bof_r(z^{-k},r)+d_r \,].
\end{equation*}
On the other hand,
\begin{align*}
&W(z^{-k},r) -  W(z^{-k},a)\\
& = \sum_{p \in A}\frac{b_p}{2} \big[(\bof_p(z^{-k},a))^2+\sum_{j\in I_2}(m^j x^{j}_p)^2\big]+\sum_{p\in A}d_{p}\bof_p(z^{-k},a) + \frac{1}{2}[b_a (m^k)^2+\sum_{l\in I_3\setminus \{k\}}b_{s^l} (m^l)^2] \\
& ~~~-\sum_{p\in A}\frac{b_p}{2} \big[(\bof_p(z^{-k},r))^2+\sum_{j\in I_2}(m^j x^{j}_p)^2\big] -\sum_{p\in A}d_{p}\bof_p(z^{-k},r) -  \frac{1}{2}[b_{r} (m^k)^2+\sum_{l\in I_3\setminus \{k\}}b_{s^l} (m^l)^2] \\
& = \sum_{p\in \{a,r\}} \Big[\frac{b_p}{2}  (\bof_p(z^{-k},a))^2 +d_{p}\bof_p(z^{-k},a) -\frac{b_p}{2}  (\bof_p(z^{-k},r))^2 -d_{p}\bof_p(z^{-k},r) \Big] + \frac{b_a (m^k)^2}{2}- \frac{b_{r}(m^k)^2}{2}\\
& = \frac{b_a}{2}  [\bof_a(z^{-k},a)]^2+ d_a  \bof_a (z^{-k},a) -  \frac{b_a}{2}  [\bof_a(z^{-k},a)-m^k]^2 - d_a   [\bof_a(z^{-k},a)-m^k] + \frac{b_a (m^k)^2}{2}\\
&~~~ + \frac{b_r}{2}  [\bof_r(z^{-k},r)-m^k]^2 + d_r  [\bof_r(z^{-k},r)-m^k]  - \frac{b_r}{2}  [\bof_r(z^{-k},r)]^2 - d_r  \bof_r(z^{-k},r) - \frac{b_{r} (m^k)^2}{2}\\
& = m^k  [b_a \bof_a(z^{-k},a)+d_a] - m^k[b_{r}\bof_r(z^{-k},r)+d_r].
\end{align*}
Thus
\begin{equation}\label{eq:ex3}
W(z^{-k},r) -  W(z^{-k},a) = m^k [u^k(z^{-k},a) - u^k(z^{-k},r)].
\end{equation}

By a multilinear extension with respect to $z$, \eqref{eq:ex1}--\eqref{eq:ex3} imply \eqref{cnd:composite_potential}.
\end{proof}

One can verify that  the potential function $W$ is concave on $X$. Therefore, $x \in X$ is a composite equilibrium of this composite congestion game if and only if it is a global maximizer of the potential function $W$.

Concerning the dynamics in this congestion game, note that Proposition  \ref{prop:pot_lya}  applies.

\section{Conclusion}
This paper first describes three frameworks with distinct categories of participants: nonatomic populations, atomic splittable and atomic non splittable players, then introduces a class of games called composite games where the three categories  coexist. We show that the static properties of the equilibria such as its characterization via variational inequalities, some conditions for the dynamics studied such as positive correlation, and the notion of potential games and dissipative games as well as their properties can be extended to composite games. In particular, the unilateral property of the dynamics and that of the positive correlation condition allow the dynamical system to converge when different participants follow different dynamics in a composite game.

As a matter of fact, one can further define a more general category of atomic players, called {\em composite players}. A composite player  of weight $m^i$ is described by the splittable component of weight $m^{i,0}$ and the non splittable components  of weight $m^{i,l}$ hence    represented by a vector $\underline{m}^{i}=(m^{i,0}, m^{i,1},\ldots,m^{i,n^{i}})$, where $n^{i}\in \nit^{*}$, $m^{i,0}\geq 0$, $m^{i,l}>0$ and $m^{i,0} + \sum^{n^{i}}_{l=1}m^{i,l}=m^{i}$. 
Player $i$  may allocate proportions of the splittable component to different choices and also  allocate different non splittable components to  different choices. However, a non splittable component cannot  be divided.

Equilibria and dynamics in this set-up will be studied in a forthcoming work.

\section*{Appendix}
\begin{proof}[Proof of Proposition  \ref{prop:dissip_lya}]
For a trajectory of dynamics $\dot{x}=\cB_{\Phi}(x)$ with initial point $x_0$,
\begin{equation*}
 \tfrac{\text{d} H}{\text{d}t}(x_t)  = \langle \nabla H(x_t), \dot{x}_t \rangle =  \langle \nabla H(x_t), \cB_{\Phi}(x_t) \rangle = \sum_{i\in I}  \langle \nabla^i H(x_t), \cB_{\Phi}^i(x_t) \rangle.
\end{equation*}
Hence we focus on $\langle \nabla H(x_t), \cB_{\Phi}(x_t) \rangle$. The subscript for time $t$ is omitted.
\medskip

\noindent (1) RD: Given  an equilibrium $x^*$, define 
$H(x) = \sum_{i \in I} h^i (x^i)$ with $h^i(x^i)= \sum_{p\in \textrm{supp}(x^{i*})}x^{i*}_p\ln \frac{x^{i*}_p}{x^i_p}$.

$H$ has a strict local minimum at $x^*$.  In fact, consider the  neighborhood of  $x^*$ in $X$ defined by $\mathcal{V}=\{x\in X,\; \textrm{supp}(x^*)\subset  \textrm{supp}(x)\}$. The concavity of $\ln x$ and Jensen's inequality imply:
\begin{align*}
h^i(x^ i) & =-\sum_{p\in \textrm{supp}(x^{i*})}x^{i*}_p\ln \frac{x^i_{p}}{x^{i*}_p}  \geq - \ln \big(\sum_{p\in \textrm{supp}(x^{i*})}x^{i*}_p\frac{x^i_{p}}{x^{i*}_p}\big) 
 \geq - \ln \big(\sum_{p\in \textrm{supp}(x)}x^i_{p}\big) =0,
\end{align*}
and the equality in both inequalities holds if and only if $x^i=x^{i*}$.

Consider a trajectory of the  RD dynamics with initial point $x_0\in \mathcal{V}$.
\begin{align*}
 \langle \nabla H(x), \cB_\Phi(x) \rangle
&  =  -\sum_{i\in I} \sum_{p\in \textrm{supp}(x^{i*})}  \frac{x^{i*}_p}{x^i_{p}} x^i_{p} \big[\Phi^i_p(x)-\langle x^i,  \Phi^i(x) \rangle \big]  \\
&  =  \sum_{i\in I}  \langle x^i-x^{i*},  \Phi^i(x)\rangle   =  \langle x - x^*, \Phi(x)\rangle .
\end{align*}
%
For $x \neq x^*$, since $\Gamma(\Phi)$ is dissipative (resp. strictly dissipative), one has $\langle x - x^*, \Phi(x) - \Phi(x^*) \rangle \leq$  (resp. $<$) $0$, which implies that $\langle x - x^*, \Phi(x) \rangle \leq$ (resp. $<$) $\langle x - x^*, \Phi(x^*)\rangle\leq 0$, i.e. $\langle \nabla H(x), \cB_\Phi(x) \rangle \leq$ (resp. $<$) $0$.
 
Therefore, $H$ is a local Lyapunov function when $\Gamma(\Phi)$ is dissipative. If $\Gamma(\Phi)$ is strictly dissipative, then $x^*$ is the unique equilibrium, and $H$ is a strict local Lyapunov function.

This result is given by Hofbauer and Sandholm \cite{HofSan09} for population games.
\medskip

\noindent (2) BNN: $H(x) = \tfrac{1}{2}\sum_{i\in I}\sum_{p\in S^{i}}\hat{\Phi}^{i}_p(x)^{2}$. 
\begin{equation*}
  \langle \nabla H(x), \cB_{\Phi}(x) \rangle  =  \sum_{j\in I}\sum_{q\in S^j} \tfrac{\partial}{\partial x^j_q} \Big[\sum_{i\in I, p\in S^i}\hat{\Phi}^{i}_p(x)^{2} \Big] \dot x^j_q.
\end{equation*}
For $i = j$,
\begin{equation*}
\tfrac{\partial}{\partial x^j_q} \hat{\Phi}^{i}_p(x)^{2}  =   2 \hat{\Phi}^{i}_p(x)  \tfrac{\partial}{\partial x^j_q} \hat{\Phi}^{i}_p(x)  = 2 \hat{\Phi}^{i}_p(x)\, \Big[\tfrac{\partial}{\partial x^j_q}{\Phi}^{i}_p(x)  -\langle x^i,  \tfrac{\partial}{\partial x^j_q}{\Phi}^{i}(x)\rangle -  \Phi^j_q(x) \Big],
\end{equation*}
and for $i \neq j$,
\begin{equation*}
\tfrac{\partial}{\partial x^j_q} \hat{\Phi}^{i}_p(x)^{2}=  2 \hat{\Phi}^{i}_p(x) \Big[\tfrac{\partial}{\partial x^j_q}{\Phi}^{i}_p(x)  -\langle x^i,  \tfrac{\partial}{\partial x^j_q}{\Phi}^{i}(x)\rangle \Big].
\end{equation*}
Thus,
\begin{align*}
 \langle \nabla H(x), \cB_{\Phi}(x) \rangle  
= & \; \sum_{j\in I}\sum_{q\in S^j}  \sum_{i\in I}\sum_{p\in S^i}  \Big[\, \hat{\Phi}^{i}_p(x)  - x^i_p  \sum_{r\in S^i}   \hat{\Phi}^{i}_r(x)  \,\Big] \, \tfrac{\partial}{\partial x^j_q}{\Phi}^{i}_p(x)  \dot x^j_q  \\
& ~~~~  \; -   \sum_{i\in I} \big[\sum_{p\in S^i} \hat{\Phi}^{i}_p(x) \big] \big[\sum_{q\in S^i}  \Phi^i_q(x)  \dot x^i_q \big] \\
 = & \; \langle \cB_{\Phi}(x), J_\Phi(x) \cB_{\Phi}(x)\rangle - \sum_{i\in I} \big[\sum_{p\in S^i} \hat{\Phi}^{i}_p(x) \big] \langle \Phi^i(x),  \cB^i_{\Phi}(x)\rangle.
\end{align*}
Since $\cB_{\Phi}(x)\in T_X(x)$ and $\Gamma(\Phi)$ is dissipative, $\langle \cB_{\Phi}(x), J_\Phi(x) \cB_{\Phi}(x)\rangle \leq 0$. Because BNN dynamics satisfies (PC),  $\langle \Phi^i(x), \cB^i_{\Phi}(x) \rangle>0$ for $x$ such that $\cB^i_{\Phi}(x)\neq 0$, hence $\langle \nabla H(x), \cB_{\Phi}(x) \rangle \leq 0$ and the equality holds if and only if $\cB_{\Phi}(x)=0$.


Therefore, $H$ is a strict Lyapunov function.

It is clear that $H(x)\geq 0$ and the equality holds if and only if for all $i \in I$ and all $p\in S^i$, $\hat{\Phi}^i_p(x)=0$, i.e. $x\in N\!E(\Phi)$.

This result is proved by Smith \cite{Smith1983b,Smi1984a} in a more general version and by Hofbauer \cite{Hofb2000} for one-population games.
\medskip

\noindent (3) Smith: $H(x) = \sum_{i \in I} \sum_{p,q\in S^i} x^i_{p} \{[\Phi^i_{q}(x)- \Phi^i_{p}(x)]^{+})^{2}$.
 
One has:
 \begin{align*}
\tfrac{\partial H(x)}{\partial x^i_p} 
& = \sum_{q\in S^i} \big([\Phi^i_ q- \Phi^i_p]^+ \big)^2 + \sum_{j\in I} \big\{\sum_{l\in S^j} x^j_l \sum_{q\in S^j} 2[\Phi^j_q- \Phi^j_l]^{+}\big(\tfrac{\partial  \Phi^j_q}{\partial x^i_p}-\tfrac{\partial  \Phi^j_l}{\partial x^i_p}) \big\} \\
& = \sum_{q\in S^i}([\Phi^i_q- \Phi^i_p]^+)^2 + 2\sum_{j\in I} \sum_{q\in S^j} \big\{\sum_{l\in S^j} x^j_l [\Phi^j_q- \Phi^j_l]^{+} - x^j_q \sum_{l\in S^j} [\Phi^j_l- \Phi^j_q]^{+} \big\}\tfrac{\partial  \Phi^j_q}{\partial x^i_p}\\
 & = \sum_{q\in S^i} ([\Phi^i_q- \Phi^i_p]^+)^2 + 2\sum_{j\in I}\sum_{q\in S^j} \dot{x}^j_q \, \tfrac{\partial  \Phi^j_q}{\partial x^i_p}.
\end{align*}
 It follows that:
\begin{align*}
 \langle \nabla H(x), \cB_{\Phi}(x) \rangle 
 &= A + 2 \langle \cB_{\Phi}(x), J_{\Phi}(x) \cB_{\Phi}(x) \rangle,
 \end{align*} 
 where:
\begin{align*} 
 A &=  \sum_{i\in I}\sum_{p,q\in S^i} ([\Phi^i_q- \Phi^i_p]^+)^2 \dot x^i_p\\
 & = \sum_{i\in I}\sum_{p,q\in S^i} \big([\Phi^i_q- \Phi^i_p]^+ \big)^2 \big\{\sum_{l\in S^i} x^i_l [\Phi^i_p- \Phi^i_l]^+ -x^i_p\sum_{l\in S^i} [\Phi^i_l- \Phi^i_p]^+  \big\}\\
 & = \sum_{i\in I} \sum_{p,l,q\in S^i} x^i_p[\Phi^ i_l- \Phi^i_p]^+ \big\{([\Phi^i_q- \Phi^i_l]^+)^2 - ([\Phi^i_q- \Phi^i_p]^+)^2 \big\}.
\end{align*} 
Recall that  $\cB_{\Phi}(x)\in T_X(x)$, thus $\langle \cB_{\Phi}(x), J_F(x) \cB_{\Phi}(x) \rangle \leq 0$ since $\Gamma(\Phi)$ is dissipative. 
Also notice that if $\Phi^i_l- \Phi^i_p>0$, then $\Phi^i_q- \Phi^i_p> \Phi^i_q- \Phi^i_l$ and thus $[\Phi^i_q- \Phi^i_p]^+ \geq [\Phi^i_q- \Phi^i_l]^+$. As a consequence, each term  in $A$  is non positive. By taking only the terms such that $q=l$, one obtains:
\begin{equation*}
 \langle \nabla H(x), \cB_{\Phi}(x) \rangle \leq - \sum_{i\in I} \sum_{p,l,q\in S^i} x^i_p ([\Phi^i_l(x)- \Phi^i_p(x)]^+)^3\leq 0.
\end{equation*}
In addition,  $\langle \nabla H(x), \cB_{\Phi}(x) \rangle =0$ if and only if for all $i \in I$ and all $p \in S^i$, either $x^i_p=0$ or $\Phi^i_{p}(x)\geq  \Phi^i_{l}(x)$ for all $l\in S^i$; equivalently $\cB_{\Phi}(x)=0$.

Therefore, $H$ is a strict  Lyapunov function.

Clearly, $H(x)\geq 0$. And the equality holds if and only if for all $i \in I$ and all $q \in S^i$, either $x^i_q=0$ or $\Phi^i_{q}(x)\geq  \Phi^i_{p}(x)$ for all $p\in S^i$; equivalently, $x\in N\!E(\Phi)$.

This result is proved by Smith \cite{Smi84} in one-population setting.
\medskip

\noindent (4) LP: Given an  equilibrium $x^*$, let  $H(x) = \frac{1}{2} \|x-x^*\|^2$. 


Recall that for all $x\in X$, $x^* - x \in T_X(x)$, thus $\langle x^* - x, \Pi_{N_X(x)}[\Phi(x)] \rangle \leq 0$. Then:
\begin{align*}
\langle \nabla H(x), \cB_\Phi(x) \rangle 
& =\sum_{i\in I} \langle x^i-x^{i*}, \Pi_{T_{X^i}(x^i)} [\Phi^i(x)] \rangle 
 =\langle x-x^*, \Pi_{T_X(x)}[\Phi(x)] \rangle \\
& =\langle x-x^*, \Phi(x)-\Pi_{N_X(x)}[\Phi(x)] \rangle 
 = \langle x-x^*, \Phi(x) \rangle - \langle x-x^*, \Pi_{N_X(x)}[\Phi(x)] \rangle \\
& \leq  \langle x-x^*, \Phi(x)-\Phi(x^*) \rangle +  \langle x-x^*, \Phi(x^*) \rangle 
 \leq 0,
\end{align*}
because $\Gamma(\Phi)$ is dissipative and $x^*$ is an equilibrium. Besides, when $\Gamma(\Phi)$ is strictly dissipative, the equality holds if and only if $x=x^*$, the unique equilibrium.

Therefore, $H$ is a global Lyapunov function when $F$ is dissipative. In addition,  $H$ is a global strict Lyapunov function when $\Gamma(\Phi)$ is strictly dissipative.

This result is proved by Nagurney and Zhang \cite{NaguZhang97} in one-population setting.
\medskip

\noindent (5) GP: $H(x) = \sup_{y\in X}L(x,y)$ with  $L(x,y)=\langle y-x, \Phi(x) \rangle - \frac{1}{2} \|y-x\|^2$, for $x,y\in X$. 

Since:
\begin{align*}
\|y-(x+\Phi(x))\|^2
& = \|y-x\|^2+\|\Phi(x)\|^2-2\langle y-x, \Phi(x)\rangle,
\end{align*}
one has  $H(x) = L(x, y^*(x))$ with  $y^*(x)=\Pi_{X}(x+\Phi(x))$. 
 By the Envelope theorem:
\begin{align*}
\nabla H(x)  =\nabla_x L(x, y^*(x)) 
 = -\Phi(x) + (y^*(x) -x) (J^{\tau}_\Phi(x)+I)
\end{align*}
so that:
\begin{align*}
\langle \nabla H(x), \cB_{\Phi}(x) \rangle & = \langle -\Phi(x) + (y^*(x) -x) (J^{\tau}_\Phi(x)+I), y^*(x)-x\rangle \\
& = \langle  (x+\Phi(x))-\Pi_{X}(x+\Phi(x)), x-\Pi_{X}(x+\Phi(x))\rangle + \langle \cB_\Phi(x), J_\Phi(x) \cB_\Phi(x) \rangle \leq 0.
\end{align*}
The first term is negative by property of $\Pi_X$. The second term is negative because $\Gamma(\Phi)$ is dissipative.

Therefore, $H$ is a global Lyapunov function. 

Note that $H$ is a global strict Lyapunov function when $\Gamma(\Phi)$ is strictly dissipative.

Finally,
\begin{align*}
H(x) 
& = \langle\Pi_{X}(x+\Phi(x)) - x, \Phi(x) \rangle - \tfrac{1}{2}\|\Pi_{X}(x+\Phi(x))-x\|^2\\
& = \tfrac{1}{2} \|\Phi(x)\|^2  - \tfrac{1}{2}\|\Pi_{X}(x+\Phi(x))-(x+\Phi(x))\|^2 \\
& = \tfrac{1}{2} \|(x+\Phi(x))-x\|^2  - \tfrac{1}{2}\|(x+\Phi(x))-\Pi_{X}(x+\Phi(x))\|^2 \geq 0.
\end{align*}
The inequality is due to the definition of projection $\Pi_X$, and the equality holds if and only if $x=\Pi_{X}(x+\Phi(x))$, i.e. $x\in N\!E(\Phi)$.

This result is proved by Pappalardo and Passacantando \cite{PappaPassa2004} in one-population setting.
\medskip

\noindent (6) BR: $H(x) = \sup_{y\in X} M(x,y)$  with $M(x,y)=\langle y-x, \Phi(x) \rangle $, for $x,y\in X$.

Let  $H(x) = M(x, \bar y (x))$ with $\bar y (x) \in B\!R(x)$. By the Envelope theorem, for any $\bar y (x)$,
\begin{equation*}
\nabla H(x) =\nabla_x M(x,\bar y (x))  
 = -\Phi(x) + (\bar y (x) -x) J^{\tau}_\Phi(x)
\end{equation*}
Hence,
\begin{align*}
\langle \nabla H, \cB_{\Phi}(x) \rangle & = \langle -\Phi(x) + (\bar y(x) -x) J^{\tau}_\Phi(x), \bar y(x)-x\rangle 
 = - H(x)  + \langle \cB_{\Phi}(x), J_{\Phi}(x) \cB_{\Phi}(x) \rangle 
 \leq 0.
\end{align*}
The second term is negative because $\Gamma(\Phi)$ is dissipative. Then the equality holds if and only if  $H(x)=0$ or, equivalently, $x\in N\!E(\Phi)$.

Therefore $H$ is  a strict Lyapunov function. 

This result is proven by Hofbauer and Sandholm \cite{HofSan09} for population games.
\end{proof} 


\begin{thebibliography}{99}

\bibitem{Bou02} 
T. Boulogne, E. Altman, O. Pourtallier and H. Kameda, Mixed equilibrium for multiclass routing game, IEEE Trans. Automat. Control, 47 (2002), 903--916. 

\bibitem{BrownvonN1950}
G.W. Brown and J. von Neumann, Solutions of games by differential equations, Ann. Math. Studies, 24 (1950), 73--79.

\bibitem{Com09}
R. Cominetti, J. Correa and N. Stier-Moses, The impact of oligopolistic competition in networks,  Oper. Res., 57 (2009), 1421--1437.

\bibitem{Daf80a}
S.C. Dafermos, Traffic equilibrium and variational inequalities, Transportation Sci., 14 (1980), 42--54.

\bibitem{DuNagu93}
P. Dupuis and A. Nagurney, Dynamical systems and variational inequalities, Ann. Oper. Res., 44 (1993), 9--42.

\bibitem{FrieszAl1994}
T.L. Friesz, D. Bernstein, N.J. Mehta, R.L. Tobin and S. Ganjalizadeh, Day-to-day dynamic network disequilibria and idealized traveler information systems, Oper. Res., 42 (1994), 1120--1136.

\bibitem{gm91}
I. Gilboa and A. Matsui, Social stability and equilibrium, Econometrica, 59 (1991), 859--867.

\bibitem{Ha88}
P.T. Harker, Multiple equilibrium behaviors on networks, Transportation Sci., 22 (1988), 39--46.

%
%
\bibitem{HM2003b}
S. Hart and A. Mas-Colell, Uncoupled dynamics do not lead to {N}ash equilibrium,  Am. Econ. Rev.,  93 (2003), 1830--1836.

%
\bibitem{Hau85}
A. Haurie and P. Marcotte, On the relationship between Nash-Cournot and Wardrop equilibria, Networks, 15 (1985), 295--308.

\bibitem{Hofb2000}
J. Hofbauer, From Nash and Brown to Maynard Smith: equilibria, dynamics and {ESS}, Selection, 1 (2000), 81--88.

\bibitem{HofSan09}
J. Hofbauer and W.H. Sandholm, Stable games and their dynamics, J. Econom. Theory, 144 (2009), 1665--1693. 	

\bibitem{HSy}
J. Hofbauer and K. Sigmund, ``Evolutionary Games and Population Dynamics'', Cambridge University Press, Cambrige, 1998.

\bibitem{Kin86}
D. Kinderlehrer and G. Stampacchia, ``An Introduction to Variational Inequalities and Their Applications'', Academic Press, New York, 1980.

\bibitem{LahkarSand2008}
R. Lahkar and W.H. Sandholm, The projection dynamic and the geometry of population games,  Games Econom. Behav., 64 (2008), 565--590.

\bibitem{Mon96}
D. Monderer and L.S. Shapley, Potential games, Games Econom. Behav., 14 (1996), 124--143.

\bibitem{Mo}
J.J. Moreau, Proximit\'e et dualit\'e dans un espace hilbertien, Bull. Soc. Math. France, 93 (1965), 273--299.

\bibitem{NaguZhang97}
A. Nagurney and D. Zhang, Projected dynamical systems in the formulation, stability analysis, and computation of fixed demand traffic network equilibria, Transportation Sci., 31 (1997), 147--158.

\bibitem{PappaPassa2002}
M. Pappalardo and M. Passacantando, Stability for equilibrium problems: from variational inequalities to dynamical systems, J. Optim. Theory Appl., 113 (2002), 567--582.

\bibitem{PappaPassa2004}
M. Pappalardo and M. Passacantando, Gap functions and {L}yapunov functions, J. Global Optim., 28 (2004), 379--385.

\bibitem{San00}
W.H. Sandholm, Potential games with continuous player sets, J. Econom. Theory, 97 (2001), 81--108.

\bibitem{San05}
W.H. Sandholm, Excess payoff dynamics and other well-behaved evolutionary dynamics,  J. Econom. Theory, 124 (2005), 149--170.

\bibitem{San09}
W.H. Sandholm, Pairwise comparison dynamics and evolutionary foundations for {N}ash equilibrium, Games, 1 (2009), 3--17.

\bibitem{San11}
W.H. Sandholm, ``Population {G}ames and {E}volutionary {D}ynamics'', MIT Press, Cambridge, MA, 2011.

\bibitem{Se}
R. Selten, ``Preispolitik der Mehrproduktenunternehmung in der Statischen Theorie'', Springer-Verlag, 1970.


\bibitem{Smi79}
M.J. Smith, The existence, uniqueness and stability of traffic equilibria, Transportation Res. Part B, 13 (1979), 295--304.

\bibitem{Smith1983b}
M.J. Smith, An algorithm for solving asymmetric equilibrium problems with a continuous cost-flow function, Transportation Res. Part B, 17 (1983), 365--371.

\bibitem{Smi84}
M.J. Smith, The stability of a dynamic model of traffic assignment -- an application of a method of Lyapunov, Transportation Sci., 18 (1984), 245--252.

\bibitem{Smi1984a}
M.J. Smith, A descent algorithm for solving monotone variational inequalities and monotone complementarity problems, J. Optim. Theory Appl., 44 (1984), 485--496.

\bibitem{Swin1993}
J.M. Swinkels, Adjustment dynamics and rational play in games, Games Econom. Behav., 5 (1993), 455--484.

\bibitem{TayJon78}
P.D. Taylor and L.B. Jonker, Evolutionary stable strategies and game dynamics, Math. Biosci., 40 (1978), 145--156.

\bibitem{TsakasVoor2009}
E. Tsakas and M. Voorneveld, The target projection dynamic, Games Econom. Behav., 67 (2009), 708--719.

\bibitem{Wan11} 
C. Wan, Coalitions in network congestion games, Math. Oper. Res., 37 (2012), 654--669.

\bibitem{Wan13}
C. Wan, Jeux de congestion dans les r\'eseaux Partie {I}. Mod\`eles et \'equilibres, Tech. Sci. Inform., 32 (2013) 951--980.

\bibitem {War52} 
G. Wardrop, Some theoretical aspects of road traffic research communication networks, Proc. Inst. Civ. Eng., Part 2, 1 (1952), 325--378.

\bibitem{YangZh2008}
H. Yang and X. Zhang, Existence of anonymous link tolls for system optimum on networks with mixed equilibrium behaviors, Transportation Res. Part B, 42 (2008), 99--112.

\bibitem{ZhangNa1997}
D. Zhang and A. Nagurney, Formulation, stability, and computation of traffic network equilibria as projected dynamical systems, J. Optim. Theory Appl., 93 (1997), 417--444.
\end{thebibliography}
\end{document}